%% file: main.tex
\documentclass{article}

\usepackage{mystyle}

\input{tex/title.tex}

\begin{document}
\maketitle

\input{tex/abstract.tex}

\input{tex/intro.tex}

\input{tex/reductions.tex}

\input{tex/discharging.tex}

\bibliographystyle{abbrv}
\bibliography{frac}

\end{document}

%% file: tex/title.tex
\title{Fractional vertex-arboricity of planar graphs}

\author{
  Marthe Bonamy%
  \thanks{Univ.\ Bordeaux, CNRS,  Bordeaux INP, LaBRI, UMR 5800, F-33400, Talence, France.
    Email: \href{mailto:marthe.bonamy@u-bordeaux.fr}{\nolinkurl{marthe.bonamy@u-bordeaux.fr}}}
  \and
  Franti\v{s}ek Kardo\v{s}%
  \thanks{Univ.\ Bordeaux, CNRS,  Bordeaux INP, LaBRI, UMR 5800, F-33400, Talence, France.
    Email: \href{mailto:fkardos@labri.fr}{\nolinkurl{fkardos@labri.fr}}}
  \and
  Tom Kelly%
  \thanks{School of Mathematics, University of Birmingham.
    Email: \href{mailto:T.J.Kelly@bham.ac.uk}{\nolinkurl{T.J.Kelly@bham.ac.uk}}. Partially supported by the EPSRC, grant no. EP/N019504/1.}
  \and
  Luke Postle%
  \thanks{Department of Combinatorics and Optimization, University of Waterloo. 
    Email: \href{mailto:lpostle@uwaterloo.ca}{\nolinkurl{lpostle@uwaterloo.ca}}
    Partially supported by NSERC under Discovery Grant No.\ 2019-04304, the Ontario Early Researcher Awards program and the Canada Research Chairs program.}
}

%% file: tex/abstract.tex
\begin{abstract}
  We initiate a systematic study of the fractional vertex-arboricity of planar graphs and demonstrate connections to open problems concerning both fractional coloring and the size of the largest induced forest in planar graphs.  In particular, the following three long-standing conjectures concern the size of a largest induced forest in a planar graph, and we conjecture that each of these can be generalized to the setting of fractional vertex-arboricity.  In 1979, Albertson and Berman conjectured that every planar graph has an induced forest on at least half of its vertices, in 1987, Akiyama and Watanabe conjectured that every bipartite planar graph has an induced forest on at least five-eighths of its vertices, and in 2010, Kowalik, Lu\v{z}ar, and \v{S}krekovski conjectured that every planar graph of girth at least five has an induced forest on at least seven-tenths of its vertices.  We make progress toward the fractional generalization of the latter of these, by proving that every planar graph of girth at least five has fractional vertex-arboricity at most $2 - 1/324$.  
\end{abstract}

%% file: tex/intro.tex
\section{Introduction}

For $k \in \mathbb N$, a \textit{$k$-arborization} of a graph $G$ is a map $\phi : V(G) \rightarrow [k]$ such that for each $i \in [k]$, the set of vertices $\{v \in V(G) : \phi(v) = i\}$ induces a forest in $G$, and the \textit{vertex-arboricity} of $G$, denoted $\va(G)$, is the smallest $k$ such that $G$ has a $k$-arborization.  The concept of a $k$-arborization is analagous to that of a proper $k$-coloring, where color classes induce forests rather than independent sets.  In this paper, we investigate the following natural ``fractional'' variant of vertex-arboricity (first studied in~\cite{YZ05}), which has comparatively received much less attention than the fractional chromatic number.

A \textit{fractional arborization} of a graph $G$ is a map $\phi$ with domain $V(G)$ such that for each $v\in V(G)$, the image $\phi(v)$ of $v$ is a measurable subset of $\mathbb R$ of Lebesgue measure one such that for each $\alpha \in \mathbb R$, the set of vertices $\{v \in V(G) : \alpha \in \phi(v)\}$ induces a forest in $G$.  For $k\in\mathbb R$, a \textit{fractional $k$-arborization} of $G$ is a fractional arborization such that each vertex $v\in V(G)$ satisfies $\phi(v) \subseteq (0, k)$, and the \textit{fractional vertex-arboricity} of $G$, denoted $\fva(G)$, is the infimum over all positive real numbers $k$ such that $G$ has a fractional $k$-arborization.

The fractional vertex-arboricity is also related to another well-studied graph invariant, the size of a largest subset of vertices of a graph $G$ that induces a forest, which we denote $a(G)$.  By standard arguments from the study of the fractional chromatic number, the fractional vertex-arboricity of every graph $G$ satisfies the following pair of inequalities:
\begin{equation}
  \label{inequality-chain}
  |V(G)| / a(G) \leq \fva(G) \leq \va(G).
\end{equation}

The analogue of~\eqref{inequality-chain} for the fractional chromatic number is the fact that every graph $G$ satisfies $|V(G)| / \alpha(G) \leq \chi_f(G) \leq \chi(G)$, where $\alpha(G)$, $\chi_f(G)$, and $\chi(G)$ are the independence number, the fractional chromatic number, and the chromatic number of $G$, respectively.  Research on the fractional chromatic number often fits at least one of the following two themes.  On the one hand, better upper bounds for $\chi_f$ have been proved when the same bound is impossible or out of reach for $\chi$ -- in fact, in the first paper on fractional coloring,  before the Four Color Theorem was proved~\cite{AH76}, Hilton, Rado, and Scott~\cite{HRS73} proved that planar graphs have fractional chromatic number strictly less than five.  On the other hand, considerable attention has been given to generalizing bounds on the independence number to the fractional chromatic number.  We can see these two themes, even just for planar graphs, in~\cite{CR18, DH18, DH20, DM17, DSV15, HT06, PU02}.  

In this paper we extend this paradigm from fractional coloring to fractional vertex-arboricity in the context of planar graphs.  Perhaps most importantly, though, as we discuss in Section~\ref{methods-subsection}, investigating the fractional vertex-arboricity offers a new strategy to possibly improve the best known lower bound on the size of the largest induced forest in a planar graph, a longstanding open problem.  Also, it is straightforward to show that every graph $G$ satisfies $\chi_f(G) \leq 2\cdot \fva(G)$, so substantial enough progress in this area could provide generalized forms of some bounds on the fractional chromatic number.

\subsection{Planar graphs}

The vertex-arboricity of planar graphs is fairly well understood.  It is straightforward to show that planar graphs have vertex-arboricity at most three, and there exist planar graphs with vertex-arboricity three.  Hakimi and Schmeichel~\cite{HS89} proved a topological characterization of planar graphs with vertex-arboricity two (generalizing a result of Stein~\cite{S71}) and showed that it is NP-complete to determine if a planar graph has vertex-arboricity at most two.  Raspaud and Wang~\cite{RW08} proved that for $k\in\{3, 4, 5, 6\}$, a planar graph without $k$-cycles has vertex-arboricity at most two.

On the other hand, for the size of the largest induced forest in planar graphs, the following three interesting conjectures have remained open for decades.
\begin{enumerate}[(1)]
\item In 1979, Albertson and Berman~\cite{AB79} conjectured that every planar graph has an induced forest on at least half of its vertices.
\item In 1987, Akiyama and Watanabe~\cite{AW87} conjectured that every bipartite planar graph has an induced forest on at least five-eighths of its vertices.
\item In 2010, Kowalik, Lu\v{z}ar, and \v{S}krekovski~\cite{KLS10} conjectured that every planar graph of girth at least five has an induced forest on at least seven-tenths of its vertices.
\end{enumerate}
That is, these conjectures state that if $G$ is planar, then $a(G) \geq |V(G)| / 2$, and moreover $a(G) \geq 5|V(G)| / 8$ if $G$ is bipartite and $a(G) \geq 7|V(G)| / 10$ if $G$ has girth at least five.  If true, these conjectures would be tight for $K_4$, the cube, and the dodecahedron, respectively.  The best kown result toward the Albertson-Berman Conjecture is Borodin's~\cite{B79} Acyclic 5-Color Theorem, which implies that $a(G) \geq 2|V(G)| / 5$ for every planar graph $G$.  The best known result toward the Akiyama-Watanabe Conjecture is due to Wang, Xie, and Yu~\cite{WXY18}, with a bound of $a(G) \geq (4|V(G)| + 3)/7$ if $G$ is bipartite and planar.  Akiyama and Watanabe's conjecture may hold more generally for triangle-free planar graphs -- the best known result toward this stronger form is a bound of $a(G) \geq 5|V(G)| / 9$ if $G$ is triangle-free and planar, proved by Le~\cite{Le18}.  For the Kowalik-Lu\v{z}ar-\v{S}krekovski Conjecture, the best known result is a bound of $a(G) \geq 2n/3$ if $G$ is planar and has girth at least five, proved independently by Kelly and Liu~\cite{KL17} and Shi and Xu~\cite{SX17}.

We conjecture that all three of these conjectures can be generalized (via the first inequality in~\eqref{inequality-chain}) to the fractional setting, as follows.  The first such conjecture generalizes the Albertson-Berman Conjecture.

\begin{conjecture}\label{planar-fva-conj}
  Every planar graph has fractional vertex-arboricity at most two.
\end{conjecture}

Although the best possible upper bound on the vertex-arboricity of planar graphs in general is three, Borodin's~\cite{B79} Acyclic 5-Color Theorem actually implies that every planar graph has fractional vertex-arboricity at most $5/2$, as follows.  Any acyclic 5-coloring of a graph (which is a proper 5-coloring with no 2-colored cycle) can be used to construct a fractional $5/2$-arborization -- assign vertices colored the first color the interval $(0, 1)$, assign vertices colored the second color the interval $(1, 2)$, assign vertices colored the third color $(2, 5/2) \cup (0, 1/2)$, assign vertices colored the fourth color the interval $(1/2, 3/2)$, and assign vertices colored the fifth color the interval $(3/2, 5/2)$.  Since every $\alpha \in (0, 5/2)$ is assigned to vertices colored one of at most two colors, this assignment is indeed a fractional arborization.  Any improvement on this upper bound of $5/2$ would be significant, since it would be the first improvement over the bound $a(G) \geq 2|V(G)|/5$ for planar graphs in over forty years.  Also, if true, Conjecture~\ref{planar-fva-conj} implies that every planar graph has fractional chromatic number at most four.  Although this result follows from the Four Color Theorem, no other proof is known (see~\cite{CR18}).

Our second conjecture generalizes the Akiyama-Watanabe Conjecture, not only to the fractional setting but also by replacing the hypothesis that the graph is bipartite with the weaker assumption that it is triangle-free.

\begin{conjecture}\label{tri-free-fva-conj}
  Every triangle-free planar graph has fractional vertex-arboricity at most $8/5$.
\end{conjecture}

Finally, we conjecture that the Kowalik-Lu\v{z}ar-\v{S}krekovski Conjecture can be generalized to the fractional setting as well.

\begin{conjecture}\label{girth-five-fva-conj}
  Every planar graph of girth at least five has fractional vertex-arboricity at most $10/7$.
\end{conjecture}

Previously, the best known bound for the fractional vertex-arboricity of triangle-free and girth at least five planar graphs was two, the same as for the vertex-arboricity~\cite{RW08}.  Our main result is to ``break the integer barrier'' for planar graphs of girth at least five and make the first significant progress toward Conjecture~\ref{girth-five-fva-conj}, as follows.

\begin{theorem}\label{main-thm}
  The fractional vertex-arboricity of every planar graph of girth at least five is at most $2 - 1/324$.
\end{theorem}

As is the case for fractional coloring, the first inequality in~\eqref{inequality-chain} holds with equality if $G$ is vertex-transitive.  The aforementioned tight examples for conjectures concerning the size of the largest induced forest are all vertex-transitive, so they all satisfy Conjectures~\ref{planar-fva-conj}--\ref{girth-five-fva-conj}.  Nevertheless, if any one of Conjectures~\ref{planar-fva-conj}--\ref{girth-five-fva-conj} turns out to be false, it would still be interesting to determine the best possible upper bound on the fractional vertex-arboricity of the graphs under consideration.  It would already be interesting to prove that planar graphs have fractional vertex-arboricity at most $5/2 - \eps$ for some $\eps > 0$ and that triangle-free planar graphs have fractional vertex-arboricity at most $2 - \eps$ for some $\eps > 0$.  It would also be interesting to obtain a significant improvement to our upper bound in Theorem~\ref{main-thm}.  There are some places in our argument where a more careful analysis would yield minor improvements to our upper bound -- we opted for the simplest proof of a bound of $2 - \eps$ for any $\eps > 0$ -- but a bound of $3/2 - \eps$ for example would imply the conjecture of Dvo\v{r}\'{a}k and Mnich~\cite{DM17} that for some $\eps > 0$, planar graphs of girth at least five have fractional chromatic number at most $3 - \eps$.  Conjecture~\ref{girth-five-fva-conj}, if true, would thus confirm this conjecture in a strong sense (in addition to confirming the Kowalik-Lu\v{z}ar-\v{S}krekovski Conjecture).  

\subsection{Methods}\label{methods-subsection}

Now we discuss our strategy for proving Theorem~\ref{main-thm} as well as its implications for Conjectures~\ref{planar-fva-conj} and~\ref{tri-free-fva-conj}.  Our proof consists of using discharging (Lemma~\ref{discharging-lemma}) to show that every planar graph of girth at least five has a configuration of vertices that is ``reducible'' (Lemmas~\ref{effective-min-degree-two-lemma} and~\ref{degree-three-adj-to-two-degree-threes-lemma}, pictured in Figures~\ref{degree-two-reducible-config-labelling} and~\ref{reducible-config-labelling}), which in this context essentially means that any fractional $(2 - 1/324)$-arborization of the graph obtained by removing these vertices can be extended to one of the whole graph.

Proofs of related results, such as~\cite{B79, Le18, WXY18}, use a similar strategy, but the notion of ``reducibility'' depends on the context.  It is notable that, compared with Borodin's Acyclic 5-Color Theorem~\cite{B79} and the Albertson-Berman Conjecture, partial progress (as in~\cite{DMP19, KL17, Le18, WXY18}) has been made toward the Akiyama-Watanabe Conjecture and the Kowalik-Lu\v{z}ar-\v{S}krekovski Conjecture by an approach that is more direct in some sense, wherein proving reducibility of a configuration involves extending an induced forest to a larger one, versus extending a vertex-coloring.  
Suppose $G$ is a graph, $X\subseteq V(G)$, and we aim to show that any induced forest of $G - X$ can be extended to one of $X$ in a nontrivial way.  If there is a vertex $v\in X$ with two neighbors not in $X$, and if these two neighbors are in the same component of the induced forest that we want to extend, then we cannot extend it to include $v$.  However, if we are extending an acyclic 5-coloring instead, then we still have a choice of at least three different colors to assign to $v$, and if we are extending a fractional $(5/2 - \eps)$-arborization, then we still have the freedom to choose from a subset of measure at least $(3/2 - \eps)$ to assign to $v$.  In this way, especially in the context of the Albertson-Berman Conjecture and Conjecture~\ref{planar-fva-conj}, proving a stronger statement may be useful for demonstrating reducibility of the configurations.

If we instead aim to extend a fractional $(2 - \eps)$-arborization of $G - X$ to one of $X$, as in Theorem~\ref{main-thm} and possibly in a proof of Conjecture~\ref{tri-free-fva-conj}, then a vertex $v\in X$ with at least two neighbors not in $X$ requires special care, because those two neighbors may be the ends of a path $P$ in which all vertices are assigned the same measurable set in the fractional arborization, in which case any subset of $(0, 2 - \eps)$ of measure one that we assign to $v$ contains a point assigned to all vertices of the path.  A standard trick would be to add an edge between these two neighbors, ensuring this situation does not happen, but then we would need to ensure that the resulting graph still satisfies the girth hypothesis.  In our proof of Theorem~\ref{main-thm} we use a different approach.  We obtain the fractional $(2 - \eps)$-arborization of $G - X$ by first finding a fractional $(2 - \eps)$-arborization of $G - u$ for some $u \in X\setminus\{v\}$ that has at most one neighbor not in $X$, and then we take its restriction to $G - X$.  In effect, we are ``recoloring'' the vertices in $X\setminus\{u\}$ in order to extend the fractional arborization to $u$.


%% file: tex/reductions.tex
\section{Reducibility}

The main results of this section are Lemmas~\ref{effective-min-degree-two-lemma} and~\ref{degree-three-adj-to-two-degree-threes-lemma}.  Each lemma implies that a certain configuration is reducible, that is, it does not appear in a minimum counterexample to Theorem~\ref{main-thm}.  First, in Section~\ref{prelim-subsection} we develop some machinery which we use to prove these lemmas.  We expect this machinery to also be useful for making further progress toward Conjectures~\ref{planar-fva-conj}--\ref{girth-five-fva-conj}.

\input{tex/preliminaries.tex}

\subsection{Reducible configurations}

In this subsection we prove Lemmas~\ref{effective-min-degree-two-lemma} and~\ref{degree-three-adj-to-two-degree-threes-lemma}.  For $k \in \mathbb Q$, we say a graph $G$ is \textit{$k$-arborically critical} if it has fractional vertex arboricity greater than $k$ but all proper subgraphs of $G$ have fractional vertex arboricity at most $k$.  A hypothetical counterexample to Theorem~\ref{main-thm} contains a $(2 - \eps)$-arborically critical subgraph for $\eps = 1/324$, and Lemmas~\ref{effective-min-degree-two-lemma} and~\ref{degree-three-adj-to-two-degree-threes-lemma} specify configurations that do not appear in such a subgraph.

For small $\eps > 0$, there is quite a bit of flexibility when it comes to assigning ``color'' to vertices of degree at most two in a fractional $(2 - \eps)$-arborization.  For that reason, a perhaps more important notion than the degree of a vertex is the number of neighbors of degree at least three that it has, which we refer to as its \textit{effective degree}.  Our first reducible configuration involves a vertex of effective degree at most two with small total degree.

\begin{lemma}\label{effective-min-degree-two-lemma}
  Let $\eps \leq 5/49$.  If $G$ is a $(2 - \eps)$-arborically critical graph and $v\in V(G)$ has effective degree at most two and a neighbor of degree two, then $d(v) \geq 10$.
\end{lemma}
\begin{proof}
  Suppose to the contrary that $v$ has effective degree at most two, a neighbor of degree two, and degree at most nine.  We assume that $d(v) = 9$, because the proof in the case $d(v) \leq 8$ is strictly easier.  Thus, since $v$ has effective degree two, $v$ has at least seven degree-two neighbors, say $u_1, \dots, u_7$.  Let $x$ and $y$ denote the other two neighbors of $v$ distinct from $u_1, \dots, u_7$.  See Figure~\ref{degree-two-reducible-config-labelling} for an illustration.

  \begin{figure}
    \centering
    \begin{tikzpicture}[scale = 2]
      \tikzstyle{internal} = [circle, fill];
      \tikzstyle{external} = [circle, draw];
      \node[internal, label=90:$v$] (v) at (0, 0) {};

      \node[external, label=90:$u'_1$] (u'1) at ($(v) + (180:1)$) {};
      \node[internal, label=90:$u_1$] (u1) at ($(v)!.5!(u'1)$) {};
      \node[external] (u'2) at ($(v) + (-150:1)$) {};
      \node[internal] (u2) at ($(v)!.5!(u'2)$) {};
      \node[external] (u'3) at ($(v) + (-120:1)$) {};
      \node[internal] (u3) at ($(v)!.5!(u'3)$) {};
      \node[external] (u'4) at ($(v) + (-90:1)$) {};
      \node[internal] (u4) at ($(v)!.5!(u'4)$) {};
      \node[external] (u'5) at ($(v) + (-60:1)$) {};
      \node[internal] (u5) at ($(v)!.5!(u'5)$) {};
      \node[external] (u'6) at ($(v) + (-30:1)$) {};
      \node[internal] (u6) at ($(v)!.5!(u'6)$) {};
      \node[external, label=90:$u'_7$] (u'7) at ($(v) + (0:1)$) {};
      \node[internal, label=90:$u_7$] (u7) at ($(v)!.5!(u'7)$) {};
      \draw (v) -- (u'1);
      \draw (v) -- (u'2);
      \draw (v) -- (u'3);
      \draw (v) -- (u'4);
      \draw (v) -- (u'5);
      \draw (v) -- (u'6);
      \draw (v) -- (u'7);

      \node[external, label=-30:$y$] (y) at ($(v) + (60:1)$) {};
      \node[external, label=-150:$x$] (x) at ($(v) + (120:1)$) {};
      \draw (v) -- (x); \draw (v) -- (y);
    \end{tikzpicture}
    \caption{Reducible configuration in Lemma~\ref{effective-min-degree-two-lemma}}
    \label{degree-two-reducible-config-labelling}
  \end{figure}
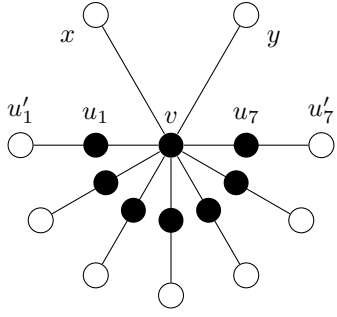

      Since $G$ is $(2 - \eps)$-arborically critical, there is a fractional $(2 - \eps)$-arborization $\phi$ of $G - \{u_1, \dots, u_7\}$ (we note that here is where we need $v$ to have at least one neighbor of degree two).  Let $H = G[\{v, u_1, \dots, u_7\}]$, let $\phi'$ be the restriction of $\phi$ to $G - V(H)$, and define the following fractional offshoot-assignment $(L, o)$ for $H$:
    \begin{itemize}
    \item let $L(v) = (0, 2 - \eps) \setminus B$, where $B$ is the union taken over all cycles $C$ in $G - \{u_1, \dots, u_7\}$ containing $v$ of $\cap_{u\in V(C - v)}\phi(u)$,
    \item let $o(v) = L(v) \cap (\phi(x) \cup \phi(y))$,
    \item for $i\in[7]$, let $L(u_i) = (0, 2 - \eps)$, and
    \item for $i\in[7]$, let $o(u_i) = \phi(u'_i)$.
    \end{itemize}
    Crucially, if $H$ has a fractional $(L, o)$-arborization, then $\phi'$ can be extended to a fractional $(2 - \eps)$-arborization of $G$, contradicting the criticality of $G$.  The remainder of the proof is thus devoted to showing that $H$ has a fractional $(L, o)$-arborization, using Proposition~\ref{offshoot-coloring-proposition}.

    First, observe that $B \cap \phi(v) = \varnothing$, so $\meas{B} + \meas{\phi(v)} \leq 2 - \eps$, and therefore
    \begin{equation}
      \label{v-has-enough-color}
      \meas{B} \leq 1 - \eps.
    \end{equation}
    Now we claim that
    \begin{equation}
      \label{v-offshoot-bound}
      \meas{o(v)} \leq 2 - 2\meas{B}.
    \end{equation}
    Since, $B \cap o(v) = \varnothing$, we have
    \begin{equation*}
      \meas{o(v)} \leq \meas{\phi(x) \cup \phi(y)} - \meas{B},
    \end{equation*}
    and since $B\subseteq \phi(x) \cap \phi(y)$, we have
    \begin{equation*}
      \meas{\phi(x) \cup \phi(y)} \leq \meas{\phi(x)} + \meas{\phi(y)} - \meas{\phi(x) \cap \phi(y)} \leq 2 - \meas{B}.
    \end{equation*}
    Combining the two inequalities above, we have $\meas{o(v)} \leq 2 - 2\meas{B}$, as claimed.

    We now define demand functions $f_{X, O}$ for every pair of disjoint subsets $X, O\subseteq V(H)$ in order to apply Proposition~\ref{offshoot-coloring-proposition}.  When $v\in O$, we define these functions differently depending on whether $\meas{B} > 1 - 7\eps/5$.  We claim there exist demand functions $f_{X,O}$ for every pair of disjoint subsets $X, O\subseteq V(H)$ such that $G - X$ has a fractional $f_{X,O}$-arborization and moreover these functions satisfy the following:
    \begin{enumerate}[(i)]
    \item\label{v-demands-all-free-color} for $i\in[7]$, if $u_i \notin X\cup O$, then $f_{X, O}(u_i) = 1$, and likewise for $v$,
    \item\label{v-demands-at-least-2/3} if $v\notin X$, then $f_{X, O}(v) \geq 6/7$.
    \item\label{v-demands-all-if-much-blocked} if $v \in O$ and $\meas{B} > 1 - 7\eps/5$, then $f_{X, O}(v) = 1$, and
    \item\label{u_i-demands-at-least-1/3} for $i\in[7]$, if $u_i\notin X$, then $f_{X, O}(u_i) \geq 1/7$, unless $v \in O$ and $\meas{B} > 1- 7\eps/5$.
    \end{enumerate}
    Now we specify these functions.  First, in all cases, if $u_i \notin X\cup O$ for $i\in[7]$, then we let $f_{X, O}(u_i) = 1$, and similarly we let $f_{X, O}(v) = 1$ if $v\notin X\cup O$.

    We first suppose that either $\meas{B} \leq 1 - 7\eps/5$ or $v \notin O$.  For $i\in[7]$, if $u_i \in O$, let $f_{X, O}(u_i) = 1/7$, and if $u_i\notin X\cup O$, let $f_{X, O}(u_i) = 1$.  If $v \in O$, let $f_{X, O}(v) = 6/7$.  To certify that $H - X$ has a fractional $f_{X, O}$-arborization, we assign $(0, 6/7)$ to $v$, we assign $(6/7, 1)$ to every $u_i \in O$ for $i\in[7]$, and we assign $(0, 1)$ to $u_i \notin X\cup O$ for $i\in[7]$.  Since $f_{X, O}$ satisfies~\ref{v-demands-all-free-color}-\ref{u_i-demands-at-least-1/3}, we may now assume $v\notin O$.  If $v\notin X$, we let $f_{X, O}(v) = 1$ as already stated, and to certify that $H - X$ has a fractional $f_{X, O}$-arborization, we assign $((i - 1)/7, i/7)$ to $u_i$ for each $i\in[7]$ (if $u_i\notin X$). Since $f_{X, O}$ satisfies~\ref{v-demands-all-free-color}-\ref{u_i-demands-at-least-1/3}, we now assume $v\in X$.  In this case, we let $f_{X, O}(u) = 1$ for all $u \in V(H - X)$, and assigning $(0, 1)$ to each vertex in $V(H - X)$ certifies there is a fractional $f_{X, O}$-arborization.  Since $f_{X, O}$ satisfies~\ref{v-demands-all-free-color}-\ref{u_i-demands-at-least-1/3}, this case is complete.

    Now suppose $v\in O$ and $\meas{B} > 1 - 7\eps/5$.  In this case, we let $f_{X, O}(v) = 1$ and $f_{X, O}(u_i) = 0$ (if $u_i\notin X$) for all $i\in[7]$.  Assigning $(0, 1)$ to $v$ certifies there is a fractional $f_{X, O}$-arborization, and $f_{X, O}$ satisfies~\ref{v-demands-all-free-color}-\ref{u_i-demands-at-least-1/3}, so the claim holds.

    Now we combine~\eqref{v-has-enough-color} and~\eqref{v-offshoot-bound} with~\ref{v-demands-all-free-color}-\ref{u_i-demands-at-least-1/3} and apply Proposition~\ref{offshoot-coloring-proposition} to show that $H$ has a fractional $(L, o)$-arborization.  It suffices to show that $v$, $u_1, \dots, u_7$ satisfy~\eqref{covering-equation}.

    We first show that $v$ satisfies~\eqref{covering-equation}.  By~\ref{v-demands-all-free-color} and~\ref{v-demands-at-least-2/3}, the left side of~\eqref{covering-equation} for $v$ is at least $\meas{L(v)\setminus o(v)} + 6\meas{o(v)}/7 = \meas{L(v)} - \meas{o(v)}/7$, and by~\eqref{v-offshoot-bound}, this quantity is at least
    \begin{equation*}
      (2 - \eps - \meas{B}) - \frac{1}{7}(2 - 2\meas{B}) = \frac{5}{7}(1 - 7\eps/5 - \meas{B}) + 1.
    \end{equation*}
    Thus, if $\meas{B} \leq 1 - 7\eps/5$, then the above quantity is at least 1, and~\eqref{covering-equation} holds, as desired.  Therefore we may assume that $\meas{B} > 1 - 7\eps/5$.  Now by~\ref{v-demands-all-free-color} and~\ref{v-demands-all-if-much-blocked}, the left side of~\eqref{covering-equation} is equal to $\meas{L(v)}$, and by~\eqref{v-has-enough-color}, $\meas{L(v)} \geq 1$, so~\eqref{covering-equation} holds, as required.

    Now we show that $u_1$ satisfies~\eqref{covering-equation}; the proof for $u_2, \dots, u_7$ is the same by symmetry.  First note that if $\meas{B} > 1 - 7\eps/5$, then by~\eqref{v-offshoot-bound}, $\meas{o(v)} \leq 14\eps/5$.  Thus, by~\ref{v-demands-all-free-color} and~\ref{u_i-demands-at-least-1/3}, the left side of~\eqref{covering-equation} is at least $(1 - \eps) + (1/7)(1 - 14\eps/5)$, the worst case being when $o(v) \subseteq o(u_1)$ and $\meas{o(v)} = 14\eps/5$.  Since $\eps \leq 5/49$, this quantity is at least 1, so~\eqref{covering-equation} holds, as required.  Therefore by Proposition~\ref{offshoot-coloring-proposition}, $H$ has a fractional $(L, o)$-arborization, so $\phi'$ can be extended to a fractional $(2 - \eps)$-arborization of $G$, contradicting the criticality of $G$.
\end{proof}

For the second reducible configuration, we introduce a distinction between \textit{light} and \textit{heavy} vertices.  A vertex is \textit{light} if it has effective degree at most three and degree at most five, and otherwise it is \textit{heavy}.

\begin{lemma}\label{degree-three-adj-to-two-degree-threes-lemma}
  Let $\eps \leq 1/324$.  If $G$ is a $(2 - \eps)$-arborically critical graph of girth at least five and $v\in V(G)$ has degree three, then $v$ has at least two heavy neighbors.
\end{lemma}
\begin{proof}
  Suppose to the contrary that $v$ has at most one heavy neighbor, so $v$ has at least two light neighbors, say $v_1$ and $v_2$.  By Lemma~\ref{effective-min-degree-two-lemma}, $v_1$ and $v_2$ have degree at least three.  Moreover, $v_1$ and $v_2$ each have at most two neighbors of degree two, and if one of them does, then it has degree five.  Henceforth we assume that $v_1$ and $v_2$ have degree five, because the proof in the case that at least one has degree three or four is strictly easier.  For $i\in\{1, 2\}$, since $v_i$ is light, it has two degree-two neighbors, which we denote $u_i$ and $w_i$, and we let $u'_i$ and $w'_i$ denote the neighbors of $u_i$ and $w_i$, respectively, that are distinct from $v_i$.  We also let $x_i$ and $y_i$ denote the two neighbors of $v_i$ distinct from $u_i$ and $v$, and we let $z$ denote the neighbor of $v$ distinct from $v_1$ and $v_2$.  See Figure~\ref{reducible-config-labelling} for an illustration.  Since $G$ has girth at least five, $\{u_1, w_1\} \cap \{u_2, w_2\} = \varnothing$ and $v_1$ and $v_2$ are not adjacent.

  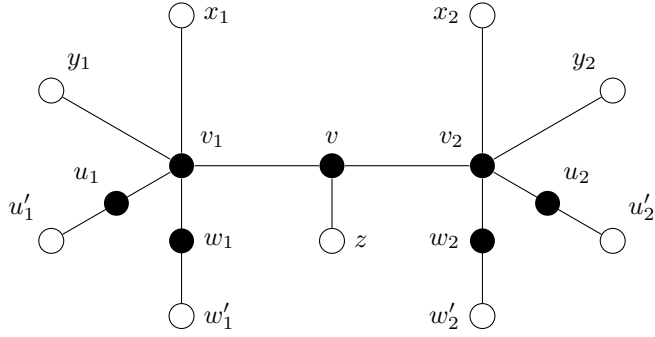
\begin{figure}
    \centering
    \begin{tikzpicture}[scale = 2]
      \tikzstyle{internal} = [circle, fill];
      \tikzstyle{external} = [circle, draw];
      \node[internal, label=90:$v$] (v) at (0, 0) {};

      \node[internal, label=45:$v_1$] (v1) at (-1, 0) {};
      \node[internal, label=135:$v_2$] (v2) at (1, 0) {};
      \draw (v1) -- (v) -- (v2);
      \node[external, label = 0:$x_1$] (x1) at ($(v1) + (90:1)$) {};
      \node[external, label = 60:$y_1$] (y1) at ($(v1) + (150:1)$) {};
      \node[external, label = 120:$u'_1$] (u'1) at ($(v1) + (-150:1)$) {};
      \node[internal, label = 120:$u_1$] (u1) at ($(v1)!.5!(u'1)$) {};
      \node[external, label = 0:$w'_1$] (w'1) at ($(v1) + (-90:1)$) {};
      \node[internal, label = 0:$w_1$] (w1) at ($(v1)!.5!(w'1)$) {};
      \draw (v1) -- (u'1); \draw (v1) -- (y1); \draw (v1) -- (x1); \draw (v1) -- (w'1);
      \node[external, label = 180:$x_2$] (x2) at ($(v2) + (90:1)$) {};
      \node[external, label = 120:$y_2$] (y2) at ($(v2) + (30:1)$) {};
      \node[external, label = 60:$u'_2$] (u'2) at ($(v2) + (-30:1)$) {};
      \node[internal, label = 60:$u_2$] (u2) at ($(v2)!.5!(u'2)$) {};
      \node[external, label = 180:$w'_2$] (w'2) at ($(v2) + (-90:1)$) {};
      \node[internal, label = 180:$w_2$] (w2) at ($(v2)!.5!(w'2)$) {};
      \draw (v2) -- (u'2); \draw (v2) -- (y2); \draw (v2) -- (x2); \draw (v2) -- (w'2);

      \node[external, label=0:$z$] (w) at (0, -.5) {};
      \draw (v) -- (w);
    \end{tikzpicture}
    \caption{Reducible configuration in Lemma~\ref{degree-three-adj-to-two-degree-threes-lemma}.}
    \label{reducible-config-labelling}
  \end{figure}

  Since $G$ is $(2 - \eps)$-arborically critical, there is a fractional $(2 - \eps)$-arborization $\phi$ of $G - v$.  First, note that for each $i\in\{1, 2\}$, we have $\meas{\phi(u_i)\cap \phi(u'_i)} \geq \eps$ and likewise for $w_i$ and $w'_i$; by letting $C$ be a measurable subset of $\phi(u_i)\cap \phi(u'_i)$ of measure at least $\eps$ and possibly reassigning $\phi(u_i)$ so that $\phi(u_i) = \left((0, 2 - \eps)\setminus \phi(u'_i)\right) \cup C$, we may assume without loss of generality that $\meas{\phi(u_i) \cap \phi(u'_i)} = \eps$, and we similarly may assume without loss of generality that $\meas{\phi(w_i) \cap \phi(w'_i)} = \eps$.  Now, let $\phi'$ be the restriction of $\phi$ to $G - \{v, v_1, v_2\}$, and define the following fractional offshoot-assignment $(L, o)$ for $G[\{v, v_1, v_2\}]$:
  \begin{itemize}
  \item for $i\in\{1, 2\}$, let $L(v_i) = (0, 2 - \eps)\setminus b(v_i)$, where $b(v_i)$ is the union taken over all cycles $C$ in $G - \{v, v_{3 - i}\}$ containing $v_i$ of $\cap_{u\in V(C - v_i)}\phi(u)$,
  \item for $i\in\{1, 2\}$, let $o(v_i) = L(v_i) \cap (\phi(x_i) \cup \phi(y_i) \cup (\phi(u_i) \cap \phi(u'_i)) \cup (\phi(w_i) \cap \phi(w'_i)))$,
  \item let $L(v) = (0, 2 - \eps)$, and
  \item let $o(v) = \phi(z)$.
  \end{itemize}
  Crucially, if $G[\{v, v_1, v_2\}]$ has a fractional $(L, o)$-arborization, then $\phi'$ can be extended to a fractional $(2 - \eps)$-arborization of $G$, contradicting the criticality of $G$.  The remainder of the proof is thus devoted to showing $G[\{v, v_1, v_2\}]$ has a fractional $(L, o)$-arborization, using Proposition~\ref{offshoot-coloring-proposition}.

  For $i\in\{1, 2\}$, let $b'(v_i) = b(v_i) \setminus ((\phi(u_i) \cap \phi(u'_i)) \cup (\phi(w_i) \cap \phi(w'_i)))$, and note that $b'(v_i) \subseteq \phi(x_i) \cap \phi(y_i)$. 
  First observe that $\phi(v_i) \cap b(v_i) = \varnothing$ for each $i\in\{1, 2\}$, so $\meas{b(v_i)} + \meas{\phi(v_i)} \leq 2 - \eps$, and therefore
  \begin{equation}
    \label{v_i-has-enough-color}
    \meas{b(v_i)} \leq 1 - \eps.
  \end{equation}

  Now we claim that for each $i\in\{1, 2\}$,
  \begin{equation}\label{v_i-offshoot-bound}
    \meas{o(v_i)} \leq 2 + 6\eps - 2 \meas{b(v_i)}.
  \end{equation}
  To deduce~\eqref{v_i-offshoot-bound}, it suffices to show that $\meas{(\phi(x_i) \cup \phi(y_i)) \cap o(v_i)} \leq 2  - 2\cdot b'(v_i)$, since $b'(v_i) \geq b(v_i) - 2\eps$ and $\meas{o(v_i)\setminus (\phi(x_i)\cup\phi(y_i))} \leq \meas{(\phi(u_i) \cap \phi(u'_i)) \cup (\phi(w_i) \cap \phi(w'_i))} \leq 2\eps$.  Now, since $b'(v_i) \cap o(v_i) = \varnothing$, we have
  \begin{equation*}
    \meas{(\phi(x_i) \cup \phi(y_i)) \cap o(v_i)} \leq \meas{\phi(x_i) \cup \phi(y_i)} - \meas{b'(v_i)},
  \end{equation*}
  and since $b'(v_i) \subseteq \phi(x_i) \cap \phi(y_i)$, we have
  \begin{equation*}
    \meas{\phi(x_i) \cup \phi(y_i)} \leq \meas{\phi(x_i)} + \meas{\phi(y_i)} - \meas{\phi(x_i) \cap \phi(y_i)} \leq 2 - \meas{b'(v_i)}.
  \end{equation*}
  Combining the two inequalities above, we obtain that $\meas{(\phi(x_i) \cup \phi(y_i)) \cap o(v_i)} \leq 2 - 2 \meas{b'(v_i)}$, as desired, so~\eqref{v_i-offshoot-bound} follows.  

  Having proved~\eqref{v_i-has-enough-color} and \eqref{v_i-offshoot-bound}, we now define demand functions $f_{X,O}$ for every pair of disjoint subsets $O, X \subseteq \{v, v_1, v_2\}$ in order to apply Proposition~\ref{offshoot-coloring-proposition}.  
  When $v_i \in O$ for some $i\in\{1, 2\}$, we define these functions differently depending on if $\meas{b(v_i)} > 1 - 17\eps$.
  We claim there exist demand functions $f_{X, O}$ for every pair of disjoint subsets $X, O\subseteq \{v, v_1, v_2\}$ such that $G - X$ has a fractional $f_{X,O}$-arborization and moreover these functions satisfy the following:
  \begin{enumerate}[(i)]
  \item\label{v_i-demands-all-of-free-color} if $v_1 \notin X\cup O$, then $f_{X, O}(v_1) = 1$, and likewise for $v_2$,
  \item\label{v_i-demands-at-least-3/5} if $v_1 \notin X$, then $f_{X, O}(v_1) \geq 3/5$, and likewise for $v_2$,
  \item\label{v_i-demands-all-if-much-blocked} if $v_1 \in O$ and $\meas{b(v_1)} > 1 - 17\eps$, then $f_{X, O}(v_1) = 1$, and likewise for $v_2$,

  \item\label{v-demands-4/5-of-free} if $v \notin X\cup O$, then $f_{X, O}(v) \geq 4/5$, unless $v_i \in O$ and $\meas{b(v_i)} > 1 - 17\eps$ for some $i\in\{1, 2\}$, and
  \item\label{v-demands-2/5-of-offshoot} if $v \in O$, then $f_{X, O}(v) \geq 2/5$, unless $v_i \in O$ and $\meas{b(v_i)} > 1 - 17\eps$ for some $i\in\{1, 2\}$.
  \end{enumerate}
  Now we specify these functions.  First, in all cases, if $v_i \notin X\cup O$ for $i\in\{1, 2\}$, then we let $f_{X, O}(v_i) = 1$.
  
  We first suppose there is no $v_i \in O$ with $\meas{b(v_i)} > 1 - 17\eps$.  If $v_1 \in O$, we let $f_{X, O}(v_1) = 3/5$.  If $v\in O$ as well, we let $f_{X, O}(v) = 2/5$.  In this case, we let $f_{X, O}(v_2) = 3/5$ if $v_2\in O$, and as already stated, $f_{X, O}(v_2) = 1$ if $v_2 \notin X\cup O$.  To certify that $G[\{v, v_1, v_2\}] - X$ has a fractional $f_{X, O}$-arborization, we assign $(0, 3/5)$ to $v_1$ and $(3/5, 1)$ to $v$.  If $v_2 \in O$, we assign $(0, 3/5)$ to $v_2$ and otherwise we assign $(0, 1)$ to $v_2$.  In either case, we have a fractional $f_{X, O}$-arborization and $f_{X, O}$ satisfies~\ref{v_i-demands-all-of-free-color}-\ref{v-demands-2/5-of-offshoot}.

  If $v_1 \in O$ and $v \notin X\cup O$, we let $f_{X, O}(v) = 4/5$.  In this case, we let $f_{X, O}(v_2) = 3/5$ if $v_2 \in O$, and as already stated, $f_{X, O}(v_2) = 1$ if $v_2\notin X\cup O$.  To certify that $G[\{v, v_1, v_2\}] - X$ has a fractional $f_{X, O}$-arborization, we assign $(0, 3/5)$ to $v_1$ and $(3/5, 1)\cup (0, 2/5)$ to $v_2$.  If $v_2 \in O$, we assign $(2/5, 1)$ to $v_2$ and otherwise we assign $(0, 1)$ to $v_2$.  In either case, we have a fractional $f_{X, O}$-arborization and $f_{X, O}$ satisfies~\ref{v_i-demands-all-of-free-color}-\ref{v-demands-2/5-of-offshoot}.

  The case $v_2 \in O$ is symmetrical, so we may assume $v_1, v_2 \notin O$.  In this case, we let $f_{X, O}(u) = 1$ for each vertex $u\in\{v, v_1, v_2\}\setminus X$, and assigning $(0, 1)$ to each vertex in $\{v, v_1, v_2\}\setminus X$ certifies there is a fractional $f_{X, O}$-arborization.  Since $f_{X, O}$ satisfies~\ref{v_i-demands-all-of-free-color}-\ref{v-demands-2/5-of-offshoot}, this case is complete.

  Now suppose a vertex $v_i \in O$ satisfies $\meas{b(v_i)} > 1 - 17\eps$.  In this case, we let $f_{X, O}(v_i) = 1$, we let $f_{X, O}(v_{3-i}) = 1$ (if $v_{3-i}\notin X$), and we let $f_{X, O}(v) = 0$ (if $v\notin X$).  Assigning $(0, 1)$ to $v_1$ and $v_2$ and assigning $\varnothing$ to $v$ certifies there is a fractional $f_{X, O}$-arborization, and $f_{X, O}$ satisfies~\ref{v_i-demands-all-of-free-color}-\ref{v-demands-2/5-of-offshoot}, as required, so the claim holds.

  Now we combine~\eqref{v_i-has-enough-color} and \eqref{v_i-offshoot-bound} with~\ref{v_i-demands-all-of-free-color}-\ref{v-demands-2/5-of-offshoot} and apply Proposition~\ref{offshoot-coloring-proposition} to show that $G[\{v, v_1, v_2\}]$ has a fractional $(L, o)$-arborization.  It suffices to show that $v$, $v_1$, and $v_2$ satisfy~\eqref{covering-equation}.

  We first show that $v_1$ satisfies~\eqref{covering-equation}; the proof is the same for $v_2$ by symmetry.
  By~\ref{v_i-demands-all-of-free-color} and~\ref{v_i-demands-at-least-3/5}, the left side of~\eqref{covering-equation} for $v_1$ is at least $\meas{L(v_1)\setminus o(v_1)} + 3\meas{o(v_1)}/5 = \meas{L(v_1)} - 2\meas{o(v_1)}/5$, and by~\eqref{v_i-offshoot-bound}, this quantity is at least
  \begin{equation*}
    (2 - \eps - \meas{b(v_1)}) - \frac{2}{5}\left(2 + 6\eps - 2\meas{b(v_1)}\right) = \frac{1}{5}\left(1 - 17\eps - \meas{b(v_1)}\right) + 1.
  \end{equation*}
  Thus, if $\meas{b(v_1)} \leq 1 - 17\eps$, then the above quantity is at least 1, and~\eqref{covering-equation} holds, as desired.  Therefore we may assume that $\meas{b(v_1)} > 1 - 17\eps$.  Now by~\ref{v_i-demands-all-of-free-color} and~\ref{v_i-demands-all-if-much-blocked}, the left side of~\eqref{covering-equation} is equal to $\meas{L(v_1)}$, and by~\eqref{v_i-has-enough-color}, $\meas{L(v_i)} \geq 1$, so~\eqref{covering-equation} holds, as required.

  Now we show that $v$ satisfies~\eqref{covering-equation}.  First note that for $i\in\{1, 2\}$, if $\meas{b(v_i)} > 1 - 17\eps$, then by~\eqref{v_i-offshoot-bound}, $\meas{o(v_i)} \leq 40\eps$.  Thus, by~\ref{v-demands-4/5-of-free} and~\ref{v-demands-2/5-of-offshoot}, the left side of~\eqref{covering-equation} for $v$ is at least $2/5 + (4/5)(1 - 81\eps)$, the worst case being when $o(v_1), o(v_2) \subseteq L(v)\setminus o(v)$ and $\meas{o(v_1)} = \meas{o(v_2)} = 40\eps$.  Since $\eps \leq 1 / 324$, this quantity is at least 1, so \eqref{covering-equation} holds, as required.  Therefore by Proposition~\ref{offshoot-coloring-proposition}, $G[\{v, v_1, v_2\}]$ has a fractional $(L, o)$-arborization, so $\phi'$ can be extended to a fractional $(2 - \eps)$-arborization of $G$, contradicting the criticality of $G$.  
\end{proof}


%% file: tex/preliminaries.tex
\subsection{Preliminaries}\label{prelim-subsection}

In this subsection we prove Proposition~\ref{offshoot-coloring-proposition}, which is an important tool in our proofs of Lemmas~\ref{effective-min-degree-two-lemma} and~\ref{degree-three-adj-to-two-degree-threes-lemma}.  If we have a graph $G$ and a fractional $k$-arborization of an induced subgraph $H$ of $G$, then this proposition provides sufficient conditions for finding a fractional $k$-arborization of $G - V(H)$ that is ``compatible'' with the one of $H$, meaning we can combine the two to obtain a fractional $k$-arborization of $G$.  In order to be able to state Proposition~\ref{offshoot-coloring-proposition}, we first introduce several concepts in the next few definitions.

The first such concept is the following fractional analogue of list coloring in the context of vertex arboricity.  Let $\mu$ be the Lebesgue measure on the real numbers.
\begin{definition}
  Let $G$ be a graph.
  \begin{itemize}
  \item If $L$ is a function with domain $V(G)$ such that $L(v)$ is a measurable subset of $\mathbb R$ for each $v\in V(G)$,
    then $L$ is a \textit{fractional list-assignment} for $G$.
  \item A \textit{fractional $L$-arborization} is a fractional arborization $\phi$ of $G$ such that every vertex $v\in V(G)$ satisfies $\phi(v)\subseteq L(v)$ and $\meas{\phi(v)} \geq 1$.
  \end{itemize}
\end{definition}

In the setting described above, we have a graph $G$ and a fractional $k$-arborization, say $\phi$, of an induced subgraph $H$ of $G$ that we are aiming to extend to a fractional $k$-arborization of $G$.  For every $v \in V(G)\setminus V(H)$, we need to avoid assigning to $v$ any $\alpha \in (0, k)$ for which there is a path $P$ in $H$ where the ends of $P$ are both adjacent to $v$ and $\alpha \in \cap_{u\in V(P)}\phi(u)$.  To that end, if we let $b(v)$ be the union taken over all such paths $P$ of $\cap_{u\in V(P)}\phi(u)$ and let $L(v) = (0, k) \setminus b(v)$ for each $v \in V(G)\setminus V(H)$, then a fractional $L$-arborization of $G - V(H)$ satisfies this requirement.

However, for every path $P'$ in $G - V(H)$, we also need to avoid assigning some $\alpha \in (0, k)$ to all vertices of $P'$ if there is a path $P$ in $H$ where $P\cup P'$ is a cycle and $\alpha \in \cap_{u\in V(P)}\phi(u)$.  The next concept provides us a way to ensure a fractional $L$-arborization of $G - V(H)$ satisfies this second requirement as well.

\begin{definition}\label{offshoot-assignment-definition}
  Let $G$ be a graph.
  \begin{itemize}
  \item A \textit{fractional offshoot-assignment} for $G$ is a pair $(L, o)$ where $L$ and $o$ are fractional list-assignments such that $o(v)\subseteq L(v)$ for every $v\in V(G)$.
  \item If $(L, o)$ is a fractional offshoot-assignment for $G$, a \textit{fractional $(L, o)$-arborization of $G$} is a fractional $L$-arborization $\phi$ of $G$ such that for each path $P$ in $G$, we have $(o(x)\cap o(y))\bigcap(\cap_{v\in V(P)}\phi(v)) = \varnothing$, where $x$ and $y$ are the ends of $P$.
  \end{itemize}
\end{definition}

Now if $G$ is a graph, $\phi$ is a fractional $k$-arborization of an induced subgraph $H$ of $G$, $L$ is a fractional list-assignment for $G - V(H)$ defined as before, and for each $v\in V(G)\setminus V(H)$, we let $o(v) = \cup_{u \in N(v)\cap V(H)}\phi(u)$, then a fractional $(L, o)$-arborization of $G - V(H)$ can be combined with $\phi$ to obtain a fractional $k$-arborization of $G$.  This fact is crucial in applying Proposition~\ref{offshoot-coloring-proposition} -- although we remark that in our application of it in Lemmas~\ref{effective-min-degree-two-lemma} and~\ref{degree-three-adj-to-two-degree-threes-lemma} we use a slightly more complicated fractional offshoot-assignment with respect to degree-two vertices.  Note that a fractional $k$-arborization of $G - V(H)$ that is ``compatible'' with $\phi$ is not necessarily an $(L, o)$-arborization, although it is necessarily an $L$-arborization.

Proposition~\ref{offshoot-coloring-proposition} thus provides sufficient conditions for a graph to have a fractional $(L, o)$-arborization.  These conditions involve the following ``local'' form of a fractional $k$-arborization.

\begin{definition}
  Let $G$ be a graph.
  \begin{itemize}
  \item A \textit{demand function} for $G$ is a function $f : V(G) \rightarrow [0, 1]\cap\mathbb Q$.
  \item If $f$ is a demand function for a graph $G$, an \textit{$f$-arborization} is a fractional arborization $\phi$ such that for every $v\in V(G)$, we have $\phi(v) \subseteq [0, 1]$ and $\meas{\phi(v)} \geq f(v)$.
  \end{itemize}
\end{definition}

Note that the fractional vertex-arboricity of $G$ is the infimum over all positive real numbers $k$ such that $G$ admits an $f$-arborization where $f(v) = 1/k$ for each $v\in V(G)$.

Our strategy for finding a fractional $(L, o)$-arborization is to partition the real line so that for each part and each vertex $v$, the part is contained entirely in one of $o(v)$, $L(v)\setminus o(v)$, and $\mathbb R\setminus L(v)$, treat each part like the $[0, 1]$-interval by ``scaling'' and find an $f$-arborization for some appropriately chosen demand function $f$, so that these $f$-arborizations, after ``scaling'', can be combined to obtain an $(L, o)$-arborization.  For each part of the real line, we need the corresponding $f$-arborization to have the following property with respect to the set of vertices $v$ such that $o(v)$ is contained in the part, like in Definition~\ref{offshoot-assignment-definition}.

\begin{definition}
  Let $O\subseteq V(G)$.  A fractional arborization $\phi$ of $G$ \textit{respects $O$} if for each path $P$ in $G$ with ends in $O$, we have $\cap_{v\in V(P)}\phi(v) = \varnothing$.  
\end{definition}

Finally, we state and prove Proposition~\ref{offshoot-coloring-proposition}.

\begin{proposition}\label{offshoot-coloring-proposition}
  Let $(L, o)$ be a fractional offshoot-assignment for $G$, and for every pair of disjoint subsets $O, X\subseteq V(G)$, let $f_{X,O}$ be a demand function for $G - X$.  If 
  \begin{itemize}
  \item $G - X$ has a fractional $f_{X,O}$-arborization respecting $O$ for every $O, X$, and
  \item every vertex $v\in V(G)$ satisfies
  \begin{equation}\label{covering-equation}
    \sum_{X\subseteq V(G)}\sum_{O\subseteq V(G)\setminus X} f_{X,O}(v)\cdot \meas{\left(\left(\cap_{u \in V(G)\setminus X} L(u)\setminus o(u)\right) \bigcap \left(\cap_{u\in O}o(u)\right)\right)\setminus \left(\cup_{u\in X}L(u)\right)} \geq 1,
  \end{equation}
  \end{itemize}
  then $G$ has a fractional $(L, o)$-arborization.
\end{proposition}
\begin{proof}
  For each $X\subseteq V(G)$ and $O\subseteq V(G)\setminus X$, let
  \begin{equation*}
    C_{X, O} = \left(\left(\cap_{u \in V(G)\setminus X} L(u)\setminus o(u)\right) \bigcap \left(\cap_{u\in O}o(u)\right)\right)\setminus \left(\cup_{u\in X}L(u)\right).
  \end{equation*}
  Since $G - X$ has a fractional $f_{X, O}$-arborization respecting $O$, there is a fractional arborization $\phi_{X, O}$ respecting $O$ such that every $v\in V(G)$ satisfies $\phi_{X, O}(v)\subseteq C_{X, O}$ and $\meas{\phi_{X, O}(v)} \geq f_{X, O}(v)\cdot \meas{C_{X, O}}$.  For each $v\in V(G)$, let $\phi(v) = \cup_{X\not\ni v}\cup_{O\subseteq V(G)\setminus X}\phi_{X, O}(v)$.  We claim that $\phi$ is a fractional $(L, o)$-arborization.

  Note that if $(X, O) \neq (X', O')$, then $C_{X, O} \cap C_{X', O'} = \varnothing$, so $\phi_{X, O}(v) \cap \phi_{X', O'}(v) = \varnothing$ for each $v\in V(G)$.   Thus, $\meas{\phi(v)}$ is at least the left side of~\eqref{covering-equation}, and $\phi$ is a fractional $L$-arborization.  Moreover, for any $X\subseteq V(G)$ and path $P$ in $G - X$, if $x$ and $y$ are the ends of $P$ and $x, y \in O$, then since $\phi_{X, O}$ respects $O$, we have $\cap_{v \in V(P)} \phi_{X, O}(v) = \varnothing$.  If one of $x$ or $y$, say $x$, is not in $O$, then $\phi_{X, O}(x) \cap o(x) = \varnothing$ by the definition of $C_{X, O}$.  In either case, we have $(o(x) \cap o(y))\bigcap (\cap_{v\in V(P)}\phi_{X, O}(v)) = \varnothing$, so $\phi$ is a fractional $(L, o)$-arborization, as required.
\end{proof}


%% file: tex/discharging.tex
\section{Discharging}

The main result of this section is the following lemma which we prove using discharging.  It implies that a planar graph of girth at least five has a vertex of degree one or a configuration from Lemma~\ref{effective-min-degree-two-lemma} or~\ref{degree-three-adj-to-two-degree-threes-lemma}.  Recall that a vertex is \textit{light} if it has effective degree at most three and degree at most five, and otherwise it is \textit{heavy}.

\begin{lemma}\label{discharging-lemma}
  If $G$ is a graph with average degree less than $10/3$, then $G$ contains either
  \begin{enumerate}[(i)]
  \item\label{discharging:degree-one} a vertex of degree at most one,
  \item\label{discharging:adj-degree-twos} a pair of adjacent vertices of degree two,
  \item\label{discharging:effective-degree-two} a vertex $v$ with effective degree at most two and degree between three and nine, or
  \item\label{discharging:degree-three-adj-to-two-light} a vertex of degree three adjacent to at least two light neighbors.
  \end{enumerate}
\end{lemma}

\begin{proof}
  For each $v\in V(G)$, we let the charge of $v$ be $ch(v) = d(v) - 10/3$.  Since $G$ has average degree less than $10/3$, we have $\sum_{v\in V(G)}ch(v) < 0$.  Now we redistribute the charges according to the following rules, and we let $ch_*$ denote the final charge:
  \begin{enumerate}[(R1), topsep = 6pt]
  \item\label{rules:degree-twos} Every vertex sends $2/3$ charge to each of its degree-two neighbors.
  \item\label{rules:degree-threes} If $v$ is a heavy vertex, then $v$ sends $1/6$ charge to every neighbor of degree three.  
  \end{enumerate}

  Since $\sum_{v\in V(G)}ch_*(v) = \sum_{v\in V(G)}ch(v) < 0$, there is a vertex $v$ with $ch_*(v) < 0$.

  We may assume $d(v) > 1$, or else~\ref{discharging:degree-one} holds, as desired.  We now claim that $d(v) < 10$.  Since $v$ sends at most $2d(v) / 3$ charge, we have $ch_*(v) \geq ch(v) - 2d(v) / 3 = d(v)/3 - 10/3$, and since $ch_*(v) < 0$, we have $d(v) < 10$, as claimed.

  Suppose $d(v) = 2$, so $v$ receives $2/3$ charge from both of its neighbors by \ref{rules:degree-twos}.  If $v$ does not have a degree-two neighbor, then $v$ sends no charge, so $ch_*(v) = ch(v) + 4/3 = 0$, contradicting that $ch_*(v) < 0$.  Thus, $v$ has a degree-two neighbor, so~\ref{discharging:adj-degree-twos} holds, as desired.

  Therefore we may assume $d(v) > 2$.  Suppose $d(v) = 3$.  We may assume $v$ has at least two heavy neighbors, or else~\ref{discharging:degree-three-adj-to-two-light} holds.  Thus, by~\ref{rules:degree-threes}, $v$ receives $1/6$ charge from each.  If $v$ does not have a degree-two neighbor, then $v$ sends no charge, so $ch_*(v) = ch(v) + 2/6 = 0$, contradicting that $ch_*(v) <  0$.  Thus, $v$ has a degree-two neighbor and thus effective degree at most two, so~\ref{discharging:effective-degree-two} holds, as desired.

  Therefore we may assume $d(v) > 3$.  Suppose $d(v) = 4$.  If $v$ has no degree-two neighbors, then $v$ sends no charge under~\ref{rules:degree-twos} and at most $4/6$ charge under~\ref{rules:degree-threes}, so $ch_*(v) \geq ch(v) - 2/3 \geq 0$, contradicting that $ch_*(v) < 0$.  Thus, $v$ has at least one degree-two neighbor, so $v$ is light and sends no charge under~\ref{rules:degree-threes}.  If $v$ has only one degree-two neighbor, then $v$ sends at most $2/3$ charge under~\ref{rules:degree-threes}, so $ch_*(v) \geq ch(v) - 2/3 \geq 0$, contradicting that $ch_*(v) < 0$.  Therefore $v$ has at least two degree-two neighbors, so~\ref{discharging:effective-degree-two} holds.

  Therefore we may assume $d(v) > 4$.  Suppose $d(v) = 5$.  If $v$ has at most one degree-two neighbor, then $v$ sends at most $2/3 + 4/6 = 4/3$ combined charge under~\ref{rules:degree-twos} and~\ref{rules:degree-threes}, so $ch_*(v) \geq ch(v) - 4/3 = 5 - 10/3 - 4/3 = 1/3$, contradicting that $ch_*(v) < 0$.  Thus, $v$ has at least two degree-two neighbors, so $v$ is light and sends no charge under~\ref{rules:degree-threes}.  If $v$ has only two degree-two neighbors, then $v$ sends at most $4/3$ charge under~\ref{rules:degree-twos}, again contradicting that $ch_*(v) < 0$.  Therefore $v$ has at least three degree-two neighbors, so~\ref{discharging:effective-degree-two} holds.

  Therefore we may assume $d(v) > 5$.  If $v$ has at most $d(v) - 3$ degree-two neighbors, then $v$ sends at most $2(d(v) - 3)/3 + 3/6$ combined charge under~\ref{rules:degree-twos} and~\ref{rules:degree-threes}, so $ch_*(v) \geq ch(v) - 2d(v) / 3 + 3/2 = d(v) / 3 - 11/6 \geq 1/6$, contradicting that $ch_*(v) < 0$.  Thus, $v$ has at least $d(v) - 2$ degree-two neighbors, so $v$ has effective degree at most two.  Since $d(v) \leq 9$,~\ref{discharging:effective-degree-two} holds.
\end{proof}

Finally, we combine Lemmas~\ref{effective-min-degree-two-lemma},~\ref{degree-three-adj-to-two-degree-threes-lemma}, and~\ref{discharging-lemma} to prove Theorem~\ref{main-thm}.

\begin{proof}[Proof of Theorem~\ref{main-thm}]
  Suppose to the contrary that there exists a planar graph of girth at least five with fractional vertex arboricity greater than $2 - 1/324$.  This graph contains a $(2 - 1/324)$-arborically critical subgraph, say $G$.  Since $G$ is $(2 - 1/324)$-arborically critical, it has minimum degree at least two, so $G$ does not satisfy outcome~\ref{discharging:degree-one} of Lemma~\ref{discharging-lemma}.  By Lemma~\ref{effective-min-degree-two-lemma}, $G$ does not satisfy outcomes~\ref{discharging:adj-degree-twos} or~\ref{discharging:effective-degree-two} of Lemma~\ref{discharging-lemma}, and by Lemma~\ref{degree-three-adj-to-two-degree-threes-lemma}, $G$ does not satisfy outcome~\ref{discharging:degree-three-adj-to-two-light}.  Since $G$ is planar and has girth at least five, Euler's formula implies that $G$ has average degree less than $10/3$, contradicting Lemma~\ref{discharging-lemma}. 
\end{proof}